\documentclass[a4paper]{amsart}
\usepackage[new]{old-arrows}
\usepackage[T1]{fontenc}
\usepackage[english]{babel}
\usepackage{kpfonts}

\usepackage{amsmath,amssymb,amsthm,enumerate,xspace}
\usepackage{comment}
\usepackage{accents}
\usepackage[colorlinks=true,citecolor=blue,linkcolor=blue]{hyperref}
\usepackage[all]{xy}
\usepackage{tikz}
\usepackage{tikz-cd}
\usetikzlibrary{arrows,arrows.meta, shapes, shapes.geometric, shapes.misc, trees, graphs, matrix,  automata, backgrounds, calendar, positioning,external}
\tikzcdset{arrow style=tikz, diagrams={>=stealth}}
\usepackage{mathrsfs}
\usepackage{cleveref}
\usepackage{graphicx} 

\numberwithin{equation}{section}

\newtheorem{thm}{Theorem}[section]
\newtheorem{cor}[thm]{Corollary}
\newtheorem{prop}[thm]{Proposition}
\newtheorem{lem}[thm]{Lemma}

\newtheorem*{openproblem*}{Problem}
\newtheorem*{quest*}{Question}
\newtheorem*{problem*}{Problem}
\newtheorem*{claim*}{Claim}

\theoremstyle{definition}

\theoremstyle{remark}
\newtheorem{rem}[thm]{Remark}

\newcommand{\bC}{\mathbb{C}}

\newcommand{\bH}{\mathbb{H}}

\newcommand{\bN}{\mathbb{N}}

\newcommand{\bQ}{\mathbb{Q}}
\newcommand{\bR}{\mathbb{R}}
\newcommand{\bS}{\mathbb{S}}

\newcommand{\bZ}{\mathbb{Z}}

\newcommand\ra{\rightarrow}
\newcommand\lra{\longrightarrow}

\newcommand\B{\mathrm{B}}
\newcommand\SO{\mathrm{SO}}
\newcommand\Gl{\mathrm{GL}}
\newcommand\Diff{\mathrm{Diff}}

\newcommand\Bun{\mathrm{Bun}}

\newcommand{\hcoker}{/\!\!/}

\newcommand{\map}{\mathrm{map}}

\makeatother
\addtocontents{toc}{\protect\setcounter{tocdepth}{1}}
\makeatletter
\let\c@equation\c@thm
\makeatother
\numberwithin{equation}{section}
\makeatletter
\let\@wraptoccontribs\wraptoccontribs
\makeatother
\title[A note on invariants of foliated $3$-sphere bundles]{A note on invariants of foliated $3$-sphere bundles}

\author[]{Nils Prigge}
\address{Stockholms universitet\\ 
Matematiska institutionen\\106 91 Stockholm}
\email{nils.prigge@math.su.se}

\begin{document}
\begin{abstract}
In this note we prove that $H^*(\B\SO(4);\bQ)$ injects into the group cohomology of $\Diff^+(\bS^{3})$ with rational coefficients. The proof is based on an idea of Nariman who proved that the monomials in the Euler and Pontrjagin classes are nontrivial in $H^*(\B\Diff_+^{\delta}(\bS^{2n-1});\bQ)$.
\end{abstract}
\maketitle
\section{Introduction}
Let $\B \mathrm{S} \Gamma_d$ be the classifying space of Haefliger structures for codimension $d$ foliations that are transversely oriented and denote by $\nu:\B\mathrm{S}\Gamma_d\ra\B\Gl^+_d(\bR)$ the map that records the normal bundle of the foliation (see \cite{Hae71} for details). We consider the space of bundle maps $\Bun(TM,\nu^*\gamma_d)$ where $\gamma_d$ denotes the universal oriented vector bundle over $\B\Gl_d^+(\bR)$. It has a $\Diff_+(M)$-action by precomposition with the differential of a diffeomorphism, and Nariman proved (cf.\ \cite{Nar17} and \cite[Cor.\,2.5]{Nar23}) that the map $\B\Diff_+^{\delta}(M)\ra \B\Diff_+(M)$ induced by the inclusion of $\Diff_+^{\delta}(M)$, the group of diffeomorphisms with the discrete topology, factors through an acyclic map 
\[\beta:\B\Diff_+^{\delta}(M)\lra \Bun(TM,\nu^*\gamma_d)\hcoker \Diff_+(M).\]
This means in particular that $\beta$ induces an isomorphism
\begin{equation}
 H^*(\B\Diff^{\delta}_+(M);\bQ)\cong H^*(\Bun(TM,\nu^*\gamma_d)\hcoker \Diff_+(M);\bQ),
\end{equation}
and under this isomorphism the map on cohomology induced by $\B\Diff^{\delta}_+(M)\to \B\Diff_+(M)$ agrees with the map induced by the projection \[p\colon\Bun(TM,\nu^*\gamma_d)\hcoker \Diff_+(M)\ra\B\Diff_+(M).\]

If a Lie group $G$ acts smoothly on $M$, there is a commutative diagram
\begin{equation}\label{Gaction}
\begin{tikzcd}
  \Bun(TM,\nu^*\gamma_{d})\hcoker G\arrow{r}\arrow{d} & \Bun(TM,\nu^*\gamma_{d}) \hcoker \Diff_+(M)\arrow{d}\\
  \B G\arrow{r} & \B\Diff_+(M)
\end{tikzcd}
\end{equation}
and Nariman shows in \cite[Sect.\ 3]{Nar23} that if $G$ is a torus acting freely on $M$, then $\Bun(TM,\nu^*\gamma_{d})$ has a $G$-fixed point so the left vertical map has a section. Hence, the map on cohomology induced by $\B G\ra \B\Diff_+(M)$ factors as 
\[H^*(\B\Diff_+(M))\lra H^*(\B\Diff^{\delta}_+(M))\lra H^*(\B G).\]
For an odd sphere one can define classes $e,p_1,\hdots,p_{n-1}\in H^*(\B\Diff_+(\bS^{2n-1})$ that pull back to the Euler and Pontrjagin classes along the map $\B\SO(2n)\ra \B\Diff_+(\bS^{2n-1})$ induced by the action of $\SO(2n)$ on $\bS^{2n-1}$. If we further restrict to the free $S^1$-action, the image of the monomials in the Euler and Pontrjagin classes in the cohomology of $\B S^1$ are nontrivial and it follows that they are nontrivial in $ H^*(\B\Diff^{\delta}_+(\bS^{2n-1}))$ \cite[Thm 1.2]{Nar23}. But this poses the question about injectivity of the map 
\begin{equation}\label{question}
 \bQ[e,p_1,\hdots,p_{n-1}]\subset H^*(\B\Diff_+(\bS^{2n-1});\bQ)\ra H^*(\B\Diff^{\delta}_+(\bS^{2n-1});\bQ).
\end{equation}
The obvious obstacle to making a statement about injectivity of \eqref{question}
is that the Krull dimension of $H^*(\B S^1;\bQ)$ is too small. Instead, we should use the action of the maximal torus $T^n\subset \SO(2n)$ on $\bS^{2n-1}$. However, $T^n$ is not acting freely on $\bS^{2n-1}$ so that one does not know whether $\Bun(T\bS^{2n-1},\nu^*\gamma_{2n-1})$ has a $T^{n}$-fixed point from Nariman's construction. The main result of this note circumvents the need to construct actual fixed points for $n=2$.
\begin{thm}\label{MAIN}
 The induced map on cohomology
 \begin{equation}
  H^*(\B T^2;\bQ) \lra H^*(\Bun(T\bS^{3},\nu^*\gamma_{3})\hcoker T^2;\bQ)
 \end{equation}
 is an injection.
\end{thm}
\begin{cor}
The subring\label{main} $\bQ[p_1,e]\subset H^*(\B \Diff_+(\bS^{3});\bQ)$ defined by the Euler and Pontrjagin class injects into $H^*(\B\Diff^{\delta}_+(\bS^{3});\bQ)$
\end{cor}
\begin{proof}
We know that $\bQ[p_1,\hdots,p_{n-1},e]\subset H^*(\B \Diff_+(\bS^{2n-1};\bQ)$ injects into $H^*(\B T^n;\bQ)$ under the map $\B T^n\ra \B\Diff_+(\bS^{2n-1})$. It then follows from commutativity of \eqref{Gaction} and Theorem \ref{MAIN} that the composite 
\[\bQ[p_1,e]\ra H^*(\B\Diff^{\delta}_+(\bS^{3}))\ra H^*(\Bun(T\bS^{3},\nu^*\gamma_{3})\hcoker \Diff_+(\bS^{3})).\]
is a injection, so that the first map is injective as well. 
\end{proof}
\begin{rem}
 In a forthcoming paper \cite{Pri24} that previously appeared as an appendix to \cite{Nar23} we proved that $H^*(\B\SO(4);\bR)$ injects into the smooth cohomology $H^*_{\mathrm{sm}}(\Diff_+(\bS^3);\bR)$ using the method developed by Haefliger \cite{Hae78}. This statement follows from Corollary \ref{main}, but the real homotopy theory computation is still interesting as it showcases an interesting relation between $p_1^2$ and a continuously varying cohomology class as pointed out by Morita (cf.\ \cite{Nar23}).
\end{rem}
The idea for the proof of Theorem \ref{MAIN} is quite simple and we learned it from \cite[Example 3.1.16]{AP93}. Given a space $X$ with a torus action so that infinitely many distinct subtori $K_i\subset T$ of codimension $1$ occur as stabilizers, the projection $p\colon X\hcoker T \ra \B T$ induces an injection on cohomology. This is because the map induced by the projection $p:X\hcoker K_i\ra \B K_i$ is an injection on cohomology (as there is a $K_i$-fixed point) and from the commutativity of the diagram
\begin{equation*}
	\begin{tikzcd}[ampersand replacement=\&,row sep=small,column sep=small]
		H^*(X\hcoker K_i) \& H^*(X\hcoker T)\arrow{l}\\
		H^*(\B K_i)\arrow{u} \& H^*(\B T)\arrow{u}\arrow{l}
	\end{tikzcd}
\end{equation*}
it follows that 
\[ \ker(H^*(\B T)\ra H^*(X\hcoker T))\subset \ker(H^*(\B T)\ra H^*(\B K_i)),\] 
where $\ker(H^*(\B T)\ra H^*(\B K_i))$ is an ideal generated by a linear polynomial (with integral coefficients) $f_i\in H^2(\B T)$. Hence, an element $x\in \ker(H^*(\B T)\ra H^*(X\hcoker T))$ is divisible by infinitely may  distinct linear polynomials $f_i$ as the tori $K_i$ are distinct and therefore $x=0$.

\medskip
So the idea of the proof is to construct bundle maps $\tau\in \Bun(T\bS^{2n-1},\nu^*\gamma_{2n-1})$ with prescribed stabilizers, which as it turns out is still difficult to do. However, for $n=2$ we are able to show that there exists infinitely many tori $K\subset T^2$ with the property that
\begin{equation*}
 H^*(\B K)\hookrightarrow H^*(\Bun(T\bS^3,\nu^*\gamma_3)\hcoker K).
 \end{equation*}
 without using actual fixed points of the action, which proves Theorem \ref{MAIN}.
\section{Main result}
For integers $m,n\in \bZ$ we denote by $c_{m,n}\colon S^1\ra T^2$ the group homomorphism defined by $c_{m,n}(\lambda)=(\lambda^m,\lambda^n)$. Codimension 1 tori are indexed by integers $(m,n)\in \bZ^2/((m,n)\sim -(m,n))$ with $\mathrm{gcd}(m,n)=1$ as the image $K_{m,n}=\mathrm{im}(c_{m,n})$. Theorem \ref{MAIN} follows from the following statement which we prove in this section.
\begin{prop}\label{InjectivityKmn}
  Let $m\in \bN$ be odd and $n=m+2$, then the map 
 \[ H^*(\B K_{m,n};\bQ)\lra H^*(\Bun(T\bS^3,\nu^*\gamma_3)\hcoker K_{m,n};\bQ)\]
 is injective.
\end{prop}
We consider the $T^2$-action on $\bS^3$ given by $(\lambda_1,\lambda_2)\cdot (z_1,z_2)=(\lambda_1\cdot z_1,\lambda_2\cdot z_2)$, and we denote by $(S^3,K_{m,n})$ the $3$-sphere with the restricted $K_{m,n}$-action. For $m=n=1$ this corresponds to the usual free $S^1$-action on $\bS^3$ which we denote by $(\bS^3,S^1)$.
\begin{lem}
 The map 
 \begin{align*}
 \pi_{m,n}\colon (\bS^3,S^1)\ra (\bS^3,K_{m,n}),\qquad \pi_{m,n}(z_1,z_2)=\frac{(z_1^m,z_2^n)}{||(z_1^m,z_2^n)||}
 \end{align*}
 is equivariant with respect to $c_{m,n}$.
\end{lem}
The pullback $\pi_{m,n}^*T\bS^3$ is an $S^1$-equivariant vector bundle over $(\bS^3,S^1)$ and there is a canonical bundle map $\bar{\pi}_{m,n}:\pi_{m,n}^*T\bS^3\ra T\bS^3$ covering $\pi_{m,n}$ which is equivariant with respect to $c_{m,n}$. Precomposition with $\bar{\pi}_{m,n}$ defines a map
\begin{equation}\label{psi}
 \psi\colon\Bun(T\bS^3,\nu^*\gamma_3)\lra \Bun(\pi^*_{m,n}T\bS^3,\nu^*\gamma_3)
\end{equation}
which is equivariant by construction with respect to $c_{m,n}^{-1}\colon K_{m,n}\ra S^1$ if the inverse exists (i.e.\ if $\gcd(m,n)=1$).
\begin{lem}\label{Qequivalence}
 The map $\psi$ in \eqref{psi} is a rational equivalence. 
\end{lem}
\begin{proof}
The tangent bundle $T\bS^3$ is (non-equivariantly) trivial and hence so is the pullback $\pi_{m,n}^*T\bS^3$. This implies that the space of bundle maps is (non-equivariantly) homotopy equivalent to $\map(\bS^3,\overline{\B \Gamma}_3)$, where $\overline{\B \Gamma}_d$ denotes the homotopy fibre of $\nu:\B\mathrm{S}\Gamma_d\ra \B\Gl_d^+(\bR)$. Since $\overline{\B \Gamma}_3$ is $4$-connected \cite{Thu74}, these mapping spaces are connected and $\psi$ corresponds to the map of mapping spaces induced by precomposition with $\pi_{m,n}$, which is a rational equivalence since $\pi_{m,n}$ is (this follows directly by inspection of the rational models of mapping spaces, for example cf.\ \cite{Ber15}).
\end{proof}
In the following, we denote by $V_{k}=(\bC,\rho_{k})$ for $k\in\bZ$ the complex $S^1$-representation given by $\rho_{k}(\lambda)(z)=\lambda^k\cdot z$. We obtain $T^2$-representations by pulling back along the group homomorphism $\Delta\colon T^2\ra S^1$ given by $\Delta(\lambda_1,\lambda_2)=\lambda_2/\lambda_1$. In the next section we show that $T\bS^3$ is $T^2$-equivariantly $V$-trivial for $V=\Delta^*V_1\oplus \bR$, where $\bR$ denotes the trivial representation (see Lemma \ref{Vtrivial}). This implies that $\pi_{m,n}^*T\bS^3$ is $c_{m,n}^*V$-trivial and we use the following simple criterion to show that it is also $\bR^3$-trivial.
\begin{lem}\label{ParametrizedIso}
 Let $G$ be a Lie group acting on $X$ and $V,W$ be two $G$-representations. There is an isomorphism of equivariant $G$-vector bundles $X\times V\cong_{G}X\times W$ if and only if there is a $G$-equivariant map $X\ra \mathrm{Iso}(V,W)$, where $\mathrm{Iso}(V,W)$ denotes the space of isomorphisms with respect to the conjugation $G$-action.
\end{lem}
\begin{prop}\label{trivial}
Let $n=m+2$, then the pullback $\pi_{m,n}^*T\bS^3$ is isomorphic to the trivial $S^1$-bundle $\bS^3\times \mathbb{R}^3 $, i.e.\,$S^1$ acts diagonally on $\bS^3\times \mathbb{R}^3 $ and trivially on $\mathbb{R}^3$.
\end{prop}
\begin{proof}
By Lemma \ref{Vtrivial} the tangent bundle $T\bS^3$ is $T^2$-equivariantly $V$-trivial for $V=\bR\oplus \Delta^*V_1$, and hence $\pi_{m,n}^*T\bS^3\cong_{S^1} \bS^3\times (\bR\oplus c_{m,n}^*\Delta^*V_1)$. The statement follows from Lemma \ref{ParametrizedIso} if we can construct an $S^1$-equivariant map 
\begin{equation}\label{trivialization}
\bS^3\lra  \text{Iso}(\bR^3,\bR\oplus c_{m,n}^*\Delta^*V_1).
\end{equation}
We identify $\text{Iso}(\bR^3,V)\cong\Gl_3(\bR)$ and denote by $D(\lambda)\in \SO(2)$ the rotation corresponding to $\lambda\in S^1$. Then under this identification the $S^1$-action is given left multiplication with 
\begin{align*}
 \left(
 \begin{array}{c c}
  1 & 0\\
  0 & D(\lambda^{n-m})
 \end{array}
 \right)\in \SO(3).
\end{align*}
If $n-m=2$ then the double cover $f\colon\bS^3\ra \SO(3)\subset \Gl_3(\bR)$ is $S^1$-equivariant with respect to this action by Lemma \ref{DoubleCover} which concludes the proof.
\end{proof}
\begin{proof}[Proof of Prop.\ \ref{InjectivityKmn}]
It follows from Proposition \ref{trivial} there is an equivariant homeomorphism $\Bun(\pi_{m,n}^*T\bS^3,\nu^*\gamma_3)\approx_{S^1}\Bun(\bS^3\times\bR^3,\nu^*\gamma_3)$. The latter has a fixed point by the same argument as in \cite[Sect.\,3]{Nar23} so that 
\[H^*(\B S^1)\ra H^*(\Bun(\pi_{m,n}^*T\bS^3,\nu^*\gamma_3)\hcoker S^1)\]
is injective. If $m\in \bN$ is odd and $n=m+2$ then $\gcd(m,n)=1$ and $c_{m,n}\colon S^1\ra K_{m,n}$ is an isomorphism, and by Proposition \ref{Qequivalence} we have a commutative diagram 
\begin{equation}
 \begin{tikzcd}[ampersand replacement=\&]
    H^*(\Bun(\pi_{m,n}^*T\bS^3,\nu^*\gamma_3)\hcoker S^1) \arrow{r}{\psi^*}\arrow[swap]{r}{\cong} \& H^*(\Bun(T\bS^3,\nu^*\gamma_3)\hcoker K_{m,n})\\
    H^*(\B S^1) \arrow{r}{\cong}\arrow[swap]{r}{H^*(c_{m,n}^{-1})} \arrow{u}\& H^*(\B K_{m,n})\arrow{u}
 \end{tikzcd}
\end{equation}
which proves the statement.
\end{proof}
\section{Two facts about \texorpdfstring{$\bS^3$}{S3}}
We prove two elementary statements about the $3$-sphere that were used in Proposition \ref{trivial}, both relying on the fact that $\bS^3$ can be identified with the group of unit quaternions.
\begin{lem}\label{Vtrivial}
 The tangent bundle $T\bS^3$ is $T^2$-equivariantly $V$ trivial for $V=\Delta^*V_1\oplus \bR $, i.e.\ there is an isomorphism of $T^2$-equivariant vector bundles $T\bS^3\cong_{T^2}\bS^3\times V$.
\end{lem}
\begin{proof}
	Choose an inner product on $\bH$ with orthonormal basis given by $1,i,j,k$ and consider $\bS^3$ as the set of unit quaternions $\{z_1+z_2j\,|\,z_i\in \bC, |z_1|^2+|z_2|^2=1\}\subset \bH$ and $T\bS^3=\{(w_1,w_2)\in \bS^3\times \bH\,|\,w_2\bot w_1\}\subset \bS^3\times\bH$. We choose a basis of $T_1\bS^3$ given by $e_1=(1,i)$, $e_2=(1,j)$ and $e_3=(1,k)$ and define a trivialization of $T\bS^3$ by 
	\[\tau\colon\bS^3\times T_1\bS^3\ra T\bS^3,\quad (p,X_1)\mapsto Dl_p(X_1)\]
	where $l_w\colon\bS^3\ra \bS^3$ denotes left multiplication with $w=z_1+z_2j\in \bS^3$ so that
	\begin{align*}
		\tau(w,e_1)&=(w,z_1i-z_2k),\\
		\tau(w,e_2)&=(w,-z_2+z_1j),\\
		\tau(w,e_3)&=(w,z_2i+z_1k).
	\end{align*}
	The tangent bundle $T\bS^3$ has a left $T^2$-action via the differential and one computes that
	\begin{align*}
		D\mu_{(\lambda_1,\lambda_2)}\tau(w,e_1)&=((\lambda_1,\lambda_2)\cdot w,\lambda_1z_1i-\lambda_2z_2k),\\
		D\mu_{(\lambda_1,\lambda_2)}\tau(w,e_2)&=((\lambda_1,\lambda_2)\cdot w,-\lambda_1z_2+\lambda_2z_1j),\\
		D\mu_{(\lambda_1,\lambda_2)}\tau(w,e_3)&=((\lambda_1,\lambda_2)\cdot w,\lambda_1z_2i+\lambda_2z_1k),
	\end{align*}
	where $\mu_{\lambda_1,\lambda_2}\colon\bS^3\ra \bS^3$ denotes action of $(\lambda_1,\lambda_2)\in T^2$.

	We need to compute $\tau^{-1} \D\mu_{\lambda_1,\lambda_2}\tau$ in order to understand the induced $T^2$-action on $\bS^3\times T_1S^1$. We see directly that 
	\begin{equation}D\mu_{(\lambda_1,\lambda_2)}\tau(w,e_1)=\tau((\lambda_1,\lambda_2)\cdot w,e_1),
	\end{equation} 
	and for the other two cases let $\psi\in [0,2\pi)$ such that $\lambda_1/\lambda_2=\cos(\psi)+i \cdot \sin(\psi)$, then a straight forward calculation shows that 
	\begin{equation}\label{temporary}
    \begin{split}
	 D\mu_{(\lambda_1,\lambda_2)}\tau(w,e_2)&=\cos(\psi)\tau((\lambda_1,\lambda_2)\cdot w,e_2)-\sin(\psi)\tau((\lambda_1,\lambda_2)\cdot w,e_3)\\
	 D\mu_{(\lambda_1,\lambda_2)}\tau(w,e_3)&=\sin(\psi)\tau((\lambda_1,\lambda_2)\cdot w,e_2)+\cos(\psi)\tau((\lambda_1,\lambda_2)\cdot w,e_3).
	\end{split}
	\end{equation}
	Finally, observe that the coefficients in \eqref{temporary} only depend on $(\lambda_1,\lambda_2)$ and not on $w$, so the induced action on $\bS^3\times T_1S^1$ is a product action. Hence, $T_1S^1$ is a $T^2$-representation and  $f\colon T_1S^1\ra \bR \oplus \Delta^*V_1 $ defined by $f(e_1)=(1,0)$, $f(e_2)=(0,1)$ and $f(e_3)=(0,i)$ is an isomorphism of $T^2$-representations.
	\end{proof}

\begin{lem}\label{DoubleCover}
	The double cover group homomorphism $f\colon\bS^3\ra \SO(3)$ is $S^1$-equivariant with respect to the $S^1$-action on $\SO(3)$ given by left multiplication with 
	\begin{align*}
	  \left(
 \begin{array}{c c}
  1 & 0\\
  0 & D(\lambda^{2})
 \end{array}
 \right)\in \SO(3).
	\end{align*}
    for $\lambda\in S^1$.
\end{lem}
\begin{proof}
	Identify $\bR^3$ with the subspace of pure quaternions $\{xi+yj+zk\,|x,y,z\in \bR\}\subset \bH$ and $\bS^3\subset \bH$ with the unit quaternions, i.e.\,$z=z_1+z_2j\in \bS^3$ and $w=ix+z_3j\in \bR^3$ for $x\in \bR$ and $z_1,z_2,z_3\in \bC$. We then define
	\begin{align*}
	f(z)(ix+jy+kz)\colon=z\cdot w\cdot \bar{z},
	\end{align*}	
	which is a group homomorphism and preserves the inner product on $\bR^3$. Observe that $zj=j\bar{z}$ for $z\in \bC$ and hence for any pure quaternion $w=ix+zj$ and $\lambda\in S^1\subset \bC$ we have that 
	\[\lambda w \bar{\lambda}=ix+\lambda zj\bar{\lambda}=ix+\lambda z\lambda j=ix+\lambda^2zj,\]
	which corresponds to the rotation by 
		\begin{align}\label{Dlambda}
	  \left(
 \begin{array}{c c}
  1 & 0\\
  0 & D(\lambda^{2})
 \end{array}
 \right)\in \SO(3).
	\end{align}
	Finally, since 
	\begin{align*}
	f(\lambda z)(w)&=\lambda z\cdot w\cdot \bar{z}\bar{\lambda}=
	\lambda f(z)(w) \bar{\lambda}, 
	\end{align*}
	this corresponds to the left multiplication of $f(z)$ by \eqref{Dlambda}.
\end{proof}	 

\subsection*{Acknowledgements} I would like to thank Sam Nariman for introducing me to this aspect of foliation theory and posing many interesting questions in our conversations. I also thank Ronno Das for numerous helpful discussions and reading this note. This research was supported by the Knut and Alice Wallenberg foundation through grant no.\ 2019.0519.

\bibliographystyle{alpha}
\bibliography{../../Bibliography/central-bib}
\end{document}